\documentclass[english,11pt]{amsart}
\usepackage[T1]{fontenc}
\usepackage{amsthm}
\usepackage{color}
\usepackage{url}
\usepackage{euscript}
\usepackage{textcomp}
\usepackage{subfig}
\usepackage{amsmath}
\usepackage{enumitem}
\usepackage{amssymb}
\usepackage{amsthm}
\usepackage[english]{babel}
\usepackage{graphicx}

\theoremstyle{definition}
\newtheorem*{definition*}{Definition}
\newtheorem{definition}{Definition}
\newtheorem{theorem}{Theorem}
\newtheorem{lemma}[theorem]{Lemma}
\newtheorem{proposition}[theorem]{Proposition}
\newtheorem{corollary}[theorem]{Corollary}
\newtheorem{remark}{Remark}
\newtheorem*{remark*}{Remark}
\newtheorem*{theorem*}{Theorem}

\def\Z{\mathbb{Z}}

\begin{document}
\title[Simplicity of spectrum for certain MCF]{Simplicity of spectrum for certain multidimensional continued fraction algorithms}

\author[C.~Fougeron]{Charles Fougeron$^\star$}
\thanks{$\star$ Max--Plank institute for Mathematics, Vivatsgasse 7, 53111 Bonn, Germany}
\author[A.~Skripchenko]{Alexandra Skripchenko$^\dagger$}
\thanks{$\dagger$ Faculty of Mathematics, National Research University Higher School of Economics, Usacheva St. 6, 119048 Moscow, Russia \textit{and}
Skolkovo Institute for Science and Technology, Skolkovo Innovation Center, 143026 Moscow, Russia}

\begin{abstract}
In the current paper we prove simplicity for the spectrum of Lyapunov exponents for triangle sequence and Selmer algorithm in dimension 3. We introduce a strategy that can be applied for a wide class of Markovian MCF.
\end{abstract}
\maketitle

\section{Introduction}

A \emph{Markovian multidimensional continued fraction algorithm}, as in Lagarias \cite{Lagarias93}, is specified by two piecewiese continuous maps:
$$f: [0,1]^d\rightarrow [0,1]^d;$$
and
$$A: [0,1]^d \rightarrow GL(d+1, \mathbb Z).$$\\

The $n$-th partial quotient matrix of $\theta\in [0,1]^d$ is denoted by $A^{(n)}(\theta)=A(f^{n-1})(\theta)$ and we define the cocycle $C^{n}(\theta)=A^{(n)}\cdots A^{1}(\theta)$.
If the algorithm is weakly convergent, the rows of matrices $C$ provide a simultaneous approximation of $\theta$.

In the present article, we are considering the case of \textit{linear simplex-splitting} MCF.
The parameter space is a d-dimensional simplex that is splitted into a finite or countable number of subsimplices $\Delta_i$, and for each point $x=(x_1,\cdots, x_n)$ in a given subsimplex $\Delta_i$ the map is defined by the formula $$f(x) = \frac {Ax} {\vert\vert Ax\vert\vert}.$$\\

We will consider two specific examples in dimension three : the Triangle Sequence, and Cassaigne algorithm (which is a different coding for dimension three Selmer algorithm). We expect our methods to work for any given linear simplex-splitting MCF. And its application to those two examples will hopefully convince the reader.

\subsection*{The Triangle Sequence.}
Defined by T.~Garrity in 2001 as an iteration of a map on a triangle which yields a sequence of nested triangles \cite{Garrity01}, the \textit{homogeneous triangle sequence} is an algorithm that is almost surely defined by
$$F: (x_1, x_2, x_3)\in\mathbb{R}^3_+ \mapsto x'=(x'_1, x'_2, x'_3),$$
where if $\{i,j,k\} = \{1,2,3\}$ and $x_i\ge x_j \ge x_k$,
\begin{equation*}
x'_i = x_i - x_j - bx_k,\
x'_j = x_j,\
x'_k = x_k,
\end{equation*}
with $b=[\frac{x_i-x_j}{x_k}].$\\
The non-homogeneous triangle sequence (a.k.a. \emph{triangle sequence}) is a renormalized version of the map $F$:
$$f(u) = \frac{F(u)}{\vert F(u) \vert}.$$

Topological ergodicity of the iterations of the algorithm was proved in \cite{Assaf05}.
Ergodicity of the algorithm, as well as weak convergence almost surely, were established in \cite{MessaoudiNogueiraSchweiger09}.

\subsection*{Cassaigne and 3-dimensional Selmer algorithm.}

Introduced by Selmer in 1961 \cite{Selmer61}, the \textit{homogeneous Selmer algorithm} is almost surely defined by,
$$F: (x_1, x_2, x_3)\in\mathbb{R}^3_+ \mapsto x'=(x'_1, x'_2, x'_3),$$
where if $\{i,j,k\} = \{1,2,3\}$ and $x_i\ge x_j \ge x_k$,
\begin{equation*}
	x'_i = x_i - x_k,\
	x'_j = x_j,\
	x'_k = x_k.
\end{equation*}
It can be checked that the subsimplex defined by $x_i < x_j + x_k$ for all $\{i,j,k\} = \{1,2,3\}$ is an invariant attractive subset of this algorithm.\\

Cassaigne proposed a different coding of this algorithm restricted to the aforementioned subset, defined by,
$$F: (x_1, x_2, x_3)\in\mathbb{R}^3 \mapsto
\left\{
\begin{array}{ll}
	(x_1-x_3, x_3, x_2) &\text{if } x_1 > x_3\\
	(x_2, x_1, x_3-x_1) &\text{if } x_3 > x_1
\end{array}
\right.
$$

The ergodicity of Selmer algorithm restricted to its attractive invariant subsimplex as well as weak convergence almost surely is established in \cite{Schweiger00}.\\

In the current paper we introduce Lyapunov exponents for these algorithms and study its properties.
Our main result is the following
\begin{theorem}\label{main}
The Lyapunov spectra of the cocycles associated to the Triangle sequence and Cassaigne algorithms are simple.
\end{theorem}

Simplicity of spectrum of a dynamical cocycle (in particular, in case of multidimensional fraction algorithms) was established in different contexts. In particular, in case of products of random matrices several results were obtained in \cite{Furstenberg63}, \cite{GoldsheidMargulis89} and \cite{GuivarchRaugi86}.
Simplicity of spectrum for Jacobi-Perron algorithm (in any dimension) was proved in \cite{BroiseGuivarch01}.

A breakthrough result was proved in 2005 by A. Avila and M. Viana who introduced the first criterion for simplicity of spectrum for dynamical cocycle (see \cite{AvilaViana07}); they used their criterion in the proof of Kontsevich--Zorich conjecture about the spectrum of Teichmüller flow on the moduli space of Abelian differentials. Later their technique was used by A. Herrera Torres who showed simplicity of spectrum for Selmer MCF (see \cite{Torres14}).

C. Matheus, M. Möller and J.-C. Yoccoz in \cite{MatheusMollerYoccoz15} developed ideas by Avila and Viana and proved a Galois-version of their criterion (see Section \ref{SS} for details). This version was used in \cite{AvilaHubertSkripchenko16} for a fully subtractive algorithm and the Rauzy gasket.

Our proof also uses a slightly modified version of the same criterion: we check that the ergodic measure has bounded distortion in the sense of \cite{AvilaViana07} and then show that there are two cocycles with a special properties, such as Galois-pinching (for both) and twisting (for the second with respect to the first one).

\begin{remark*}
Results proved in \cite{Torres14} and \cite{BroiseGuivarch01} are more general than our statement since in both cases the simplicity of spectrum of the algorithms of any dimension is established, while our algorithms are only defined in dimension three. However, if the dimension is fixed, the strategy we suggest can be applied for a very wide class of algorithms (basically, for any ergodic Markovian algorithm with integrable cocycle) and allows to get a elementary and straight-forward proof.
\end{remark*}

\subsection{Acknowledgments} We heartily thank Carlos Matheus who explained us how to modify the proof of theorem 2.17 in \cite{MatheusMollerYoccoz15} to the case of $SL(n,\Z)$ and some other important remarks. We also thank Valérie Berthé, Pascal Hubert and Vincent Delecroix for several useful discussions.
The work was done while the first author was visiting the Max-Planck-Institut for Mathematics in Bonn, he is grateful to MPI for the excellent working conditions.
The second author was partially supported by RFBR-CNRS grant No.~18-51--15010.

\section{Special Acceleration of Ergodic MCF and its properties}
\subsection{Symbolic dynamics}

A non-homogeneous version of a linear simplex-splitting MCF defines a topological Markov shift over some countable alphabet.
One can associate with this shift a graph that we call \emph{the Rauzy diagram} using terminology from Teichmüller dynamics and interval exchange transformations: vertices of this graph are letters of the alphabet, and two vertices $a$ and $b$ are connected by a directed arrow if and only there exist points $x\in [a]$ and $y\in [b]$ such that $f(x)=y$.
First, we define a notion of a positive loop on the Rauzy graph.

\begin{definition}
A path $\gamma$ on the Rauzy diagram is called \emph{positive} if the corresponding matrix of the cocycle $A_\gamma$ has only positive entries.
\end{definition}

\begin{definition}
A path $\gamma$ is called a \emph{loop} if it starts and finishes in the same vertex of the Rauzy diagram.
\end{definition}

\begin{remark}\label{acceleration}
	Positiveness and ergodicity imply that one can consider the following acceleration of the given algorithm that was first defined in \cite{AvilaGouezelYoccoz06} for the Veech flow and interval exchange transformations (see Section 4.1.3): a \emph{special acceleration} is a first return map to some subsimplex $\Delta_1$ compactly contained in the parameter space $\Delta$ (here $\Delta_1$ = $\Delta_{\gamma_*}$, where $\gamma_*$ is some strictly positive path). Naturally, not all of the orbits of the original algorithm will return to $\Delta_1$; nevertheless, due to ergodicity, it will happen to the orbits of almost every points.
\end{remark}

\section{Uniformly expanding maps}
The main definition that we use comes from \cite{AvilaGouezelYoccoz06}:
\begin{definition}
	\label{uniformlyexp}
Let $L$ be a finite or countable set, let $\Delta$ be a parameter space, and let $\{\Delta^{(l)}\}_{(l\in L)}$ be a partition into open sets of a full measure subset of $\Delta.$
A map $Q: \cup_{l} \Delta^{(l)} \rightarrow \Delta$ is a \emph{uniformly expanding} map if:
\begin{enumerate}
\item For each $l$, $Q$ is a $C^{1}$ diffeomorphism between $\Delta^{(l)}$ and $\Delta$, and there exist constants $k>1$ (independent of $l$) and $C_{(l)}$ such that for all $x\in \Delta^{(l)}$ and all $v\in Q_{x}\Delta, k||v||\le||DQ(x)v||\le C_{(l)}||v||.$
\item Let $J(x)$ be the inverse of the Jacobian of $Q$ with respect to Lebesgue measure. Denote by $\EuScript{H}$ the set of inverse branches of $Q$. The function $log J$ is $C^{1}$ on each set $\Delta^{(l)}$ and there exists $C>0$ such that, for all $h\in \EuScript{H}$,
$$||D((logJ)\circ h)||_{C^{0}(\Delta)}\le C.$$
\end{enumerate}
\end{definition}

\begin{proposition}\label{expanding}
A special acceleration of any ergodic simplex-splitting MCF is uniformly expanding.
\end{proposition}
\begin{proof}
	The proof repeats verbatim the first part of the proof of Lemma 4.4  in \cite{AvilaGouezelYoccoz06} (see also \cite{SkripchenkoTroubetzkoy18} for more detailed proof).
	It relies on the fact that we consider a special acceleration (defined in Remark \ref{acceleration}).
	Condition (1) of definition \ref{uniformlyexp} comes from the fact that each path in the accelerated version ends with the same positive loop $\gamma_*$, thus the cocycle matrix is a product of two matrices $A_{\gamma_*} A_0$ where $A_0$ is weakly contracting (is not expanding) and $A_{\gamma_*}$ is strongly contracting (contraction coefficient is strongly larger than 1) with respect to Hilbert metric and is common to all accelerated path.
	Condition (2) is checked by computing the log of the Jacobian of the map $h: x \to Ax/||Ax||$, which is $d$-Lipschitz with respect to the Hilbert diameter of the simplex $\Delta_1$.
\end{proof}

\begin{remark*}
As we mentioned above, special acceleration can be considered as the first return map to the subsimplex compactly embedded in a positive cone. Therefore, the Hilbert metric in the proof of Proposition \ref{expanding} is a finite one; so, uniform expanding with respect to this metric implies uniform expanding with respect to Euclidean metric.
\end{remark*}

Ergodic properties of uniformly expanding maps are well-known (see Section 2 in \cite{AvilaGouezelYoccoz06}, Theorem 1.3 in \cite{Mane87} and Section 4 in \cite{Aaronson97}):
\begin{proposition}\label{measure}
Any uniformly expanding map admits a unique absolutely continuous invariant measure; this measure is ergodic, and, moreover, mixing.
\end{proposition}

Moreover, this ergodic measure has bounded distortion in a sense of Avila--Viana:

\begin{definition}
	An invariant measure $\mu$ on a topological Markov shift $\Sigma_M$ has the \emph{bounded distortion property} if there exists a positive constant
	$C(\mu)$ such that for any (non-empty) cylinder set $[a_{i_1} \cdots a_{i_n}]$ and any $1 \le j \le n$ we have
	$$
	\frac{1}{C(\mu)}  \le \frac{\mu([a_{i_1} \dots a_{i_n}])}{\mu([a_{i_1}
	\dots a_{i_j}]) \mu([a_{i_{j+1}} \dots a_{i_n}])}
	\le C(\mu).
	$$
\end{definition}

We recall that this property follows from the uniformly expanding property:
\begin{lemma}
The absolutely continuous ergodic measure of a special acceleration of a simplex-splitting MCF has bounded distortion.
\end{lemma}
\begin{proof}
	It was observed by Avila and Viana in \cite{AvilaViana07} (Appendix A) that the bounded distortion property of the measure is equivalent to the property of product structure and the last property is implied by expanding property of the map (Lemma 7.2 in \cite{AvilaViana07}) that was proved above in Proposition \ref{expanding}.
\end{proof}

\section{Lagarias conditions}
Following \cite{Lagarias93}, we define a list of properties of the Markovian multidimensional fraction algorithm that are sufficient to define Lyapunov exponents using Oseledets theorem and to check that the convergence rate of the algorithm can be estimated in terms of these exponents. We verify these properties for a special acceleration of triangle sequence and Cassaigne algorithms and discuss when do they hold for the special accelerations of other algorithms.

\subsection{Property H1: Ergodicity}
It follows from Proposition \ref{measure} that the invariant measures for the accelerated triangle sequence and Cassaigne algorithm are absolutely continuous with respect to Lebesgue measure and their density are bounded from above and from zero.

\subsection{Property H2: Covering property}
The map $f$ is piecewise continuous with non-vanishing Jacobian almost everywhere. This property obviously holds for all simplex-splitting algorithms.

\subsection{Property H3: Semi-weak convergence}
It was mentioned in \cite{Lagarias93} that it is enough to show that MCF is mixing with respect to the invariant measure, which follows from Proposition \ref{measure}.

More explicitly, it can be checked as follows: since the cylinders of the Markov partition are all simplices whose vertices are given by the rows of a convergence matrix it is sufficient to show that
the diameters of the corresponding cylinders of the associated Markov partition decreases geometrically for the set of full measure (see Section 6 in \cite{Lagarias93} for the details). The property follows from
the positivity of the matrix of the accelerated algorithm (as well as from bounded distortion property).

\subsection{Property H4: Boundedness}
This property is the standard log-integrability of the cocycle that is used in Oseledets theorem. Our proof is similar to the results in \cite{Lagarias93} for Selmer and Jacobi-Perron algorithms, and concern the slow version of the algorithm:
\begin{lemma}
The cocycle of triangle sequence is log-integrable:
$$\int_{\Delta}\log(\max(1, \vert{A}\vert))d\mu< \infty.$$
\end{lemma}
\begin{proof}
The statement follows from a direct calculation: for each step of the algorithm the matrix norm grows with linearly with a parameter $b$:
$$\vert{A}\vert = b+4$$
and the measure (area) of the corresponding subsimplex decreases quadratically with the same parameter $b$. Indeed, using renormalization condition $x_i+x_j+x_k=1$ we exclude one coordinate (say, $x_2$) and so it is easy to see that subsimplex $\Delta_i$ is a triangle with the following vertices: $(\frac{1}{2},0)$; $(\frac{b+1}{b+3}, \frac{1}{b+3})$; $(\frac{b+2}{b+4}, \frac{1}{b+4})$. Therefore $$\mu(\Delta_i)=\frac{1}{4(b+3)(b+4)}.$$
The statement about log-integrability follows from the convergence of the series $\Sigma_{n}\frac{\log n}{n^2}$.
\end{proof}

\begin{lemma}
The cocycle of Cassaigne algorithm is log-integrable.
\end{lemma}
\begin{proof}
	It is clear since the cocycle take only two finite values on two parts of the simplex.
\end{proof}

The result for the acceleration follows immediately (see \cite{AvilaViana07}, page 46): one can check that the induced cocycle is log-integrable with respect to induced measure if the original cocycle is log-integrable with respect to the original measure (and the Lyapunov exponents get multiplied by the inverse of the measure of the subsimplex).

\subsection{Property H5: Partial Quotient Mixing}
All available quotient matrices are non-negative. For any $\theta\in[0,1]^d$ where $d$ is a dimension of the algorithm we set
$\nu(\theta)=\min\{k: C^{(k)}(\theta)= A^{(k)}(\theta)\cdots A^{(1)}(\theta)\  \text{is strictly positive}\}$ and $\nu(\theta)=\infty$ if such a $k$ does not exist.

The condition is the following: $\int_{\Delta}\nu(\theta)d\lambda<\infty$, where $\lambda$ is a Lebesgue measure.\\

Note that in the triangle sequence algorithm, if we start with $x_i>x_j>x_k$, after the first step we get $x'_i<x'_k<x'_j$.
Now a direct calculation shows that the matrix becomes strictly positive already after 4 steps of the algorithm.

\begin{remark*}
Properties H1--H3 are satisfied for every special acceleration of any ergodic algorithm.
Only property H4 need to be checked on the algorithm before acceleration, but has already been proven for a large number of MCF.
\end{remark*}

\subsection{Lyapunov exponents and Convergence}
It was shown by Lagarias in \cite{Lagarias93} (see Theorem 4.1) that is conditions H1--H4 are satisfied, then the convergence rate of the algorithm for almost all the parameters can be estimated in the following way:

$$\eta(\theta)\ge 1-\frac{\lambda_2}{\lambda_1}.$$ Here $\eta(\theta)$ is the best uniform approximation exponent.

Moreover, if H5 is also satisfied, then there is a set of Lebesgue measure one for which the following equality holds for the uniform approximation exponent: $$\eta^{*}(\theta)= 1-\frac{\lambda_2}{\lambda_1}.$$

This gives a interesting application to our main result.
\begin{corollary}
	The triangle sequence and Cassaigne algorithm have strictly positive best uniform approximation exponents. Moreover the triangle sequence algorithm cannot be an optimal algorithm.
\end{corollary}

\begin{proof}
	The first statement if straightforward. For the second statement, one has to remark that the sum of Lyapunov exponents is equal to zero, since the matrices of the cocycle have determinant one.
	As proved by Lagarias, the uniform approximation exponent is always bounded by $1 + 1/d$, thus an algorithm satisfying H5 cannot be optimal, \textit{i.e.} to be an equality case for this bound, unless $\lambda_2 = \dots = \lambda_d$.
\end{proof}

\section{Simplicity of Spectrum}\label{SS}
In this section we prove our main result (Theorem \ref{main}). First, we show simplicity for the spectrum of accelerated version of the triangle sequence; as was mentioned above, Avila and Viana showed \cite{AvilaViana07} that simplicity of spectrum of the induced cocycle implies the same property for the original one.

The proof is based on the so-called Galois-version of simplicity criterion proved by J.-C. Yoccoz, C. Matheus and M. Möller in 2015 (see \cite{MatheusMollerYoccoz15}). This work developed the ideas suggested in \cite{AvilaViana07}.

Avila and Viana showed that (see also Theorem 2.13 \cite{MatheusMollerYoccoz15})
\begin{theorem}\label{AV}
Let $\mu$ be an $T$-invariant probability measure on the phase space $\Sigma$ of the Markov shift $T$  with the bounded distortion property. Let $A$ be a locally constant integrable $\mathbb G$-valued cocycle ($\mathbb G$ is, for example,  $GL(d,\mathbb R)$). Assume that $A$ is pinching and twisting. Then, the Lyapunov spectrum of $(T, A)$ with respect to $\mu$ is simple.
\end{theorem}

Precise definitions of twisting and pinching properties can be found in \cite{AvilaViana07}.
We use a slightly different notion introduced in \cite{MatheusMollerYoccoz15}:
\begin{definition}
The element $A\in SL(3,\mathbb Z)$ is called \emph{Galois-pinching} if its characteristic polynomial is irreducible over $\mathbb Q$, all its roots are real and the Galois group is the largest possible (namely, $S_3$).
\end{definition}
It was shown in \cite{MatheusMollerYoccoz15} that Galois-pinching property of a given matrix $A$ of the cocycle implies pinching of the cocyle in a sense of \cite{AvilaViana07} (see Proposition 4.2, the proof works verbatim for the case of $SL(n,\mathbb Z)$).

Now, it was also proved in \cite {MatheusMollerYoccoz15} that a cocycle is \emph{twisting} if it admits an element $B$ satisfying the following conditions: $B$ is Galois-pinching, and $A$ and $B$ don't share a common proper invariant subspace (it follows from Theorem 4.6 and an argument used in Theorem 5.4 to replace $B^2$ by $B$).

We already checked that our measure has bounded distortion. Therefore, it is enough to find two explicit paths for which the cocycles are Galois-pinching and one is twisting with respect to the other.  Will we see that it is enough that their two discriminants are positive integers which are not a square and that they are coprime.

\subsection*{Triangle sequence}
Consider the following paths in the triangle sequence algorithm, coded by the order of the coordinates, $(i,j,k)$ if $x_i < x_j < x_k$, and the number $b$:
\begin{align*}
	\gamma_1 &= ((1, 2, 3), 0) \rightarrow ((3, 1, 2), 1) \rightarrow ((2, 3, 1), 2),\\
	\gamma_2 &= ((1, 3, 2), 0) \rightarrow ((2, 1, 3), 0) \rightarrow ((3, 2, 1), 3)
\end{align*}
Along these paths the cocycles are equal to,
\begin{align*}
A^{\gamma_1} &=
\left(\begin{array}{rrr}
1 & 0 & 0 \\
0 & 1 & 0 \\
0 & 1 & 1
\end{array}\right)
\left(\begin{array}{rrr}
1 & 0 & 0 \\
1 & 1 & 1 \\
0 & 0 & 1
\end{array}\right)
\left(\begin{array}{rrr}
1 & 2 & 1 \\
0 & 1 & 0 \\
0 & 0 & 1
\end{array}\right)  =
\left(\begin{array}{rrr}
1 & 2 & 1 \\
1 & 3 & 2 \\
1 & 3 & 3
\end{array}\right) \\
A^{\gamma_2} &=
\left(\begin{array}{rrr}
1 & 0 & 0 \\
0 & 1 & 1 \\
0 & 0 & 1
\end{array}\right)
\left(\begin{array}{rrr}
1 & 0 & 0 \\
0 & 1 & 0 \\
1 & 0 & 1
\end{array}\right)
\left(\begin{array}{rrr}
1 & 1 & 4 \\
0 & 1 & 0 \\
0 & 0 & 1
\end{array}\right)  =
\left(\begin{array}{rrr}
1 & 1 & 4 \\
1 & 2 & 5 \\
1 & 1 & 5
\end{array}\right)
\end{align*}
their characteristic polynomials are,
\begin{align*}
	P^{\gamma_1} &= x^{3} - 7x^{2} + 6x - 1,\\
	P^{\gamma_2} &= x^3 - 8x^2 + 7x - 1,
\end{align*}
which discriminants are equal to,
\begin{align*}
	\Delta(P^{\gamma_1}) &= 697 = 17 \cdot 41,\\
	\Delta(P^{\gamma_2}) &= 257.
\end{align*}

\subsection*{Cassaigne algorithm}
We code the path by words in numbers $1$, $2$, weather they satisfy the first of second case in the definition of the function at each step.
Consider the paths,
\begin{align*}
	\gamma_1 &= 2 \rightarrow 1 \rightarrow 2 \rightarrow 2 \rightarrow 1\\
	\gamma_2 &= 1 \rightarrow 2 \rightarrow 2 \rightarrow 2 \rightarrow 1 \rightarrow 2 \rightarrow 1
\end{align*}
Along these paths the cocycles are equal to,
\begin{align*}
A^{\gamma_1} &=
\left(\begin{array}{rrr}
1 & 2 & 1 \\
1 & 1 & 1 \\
1 & 2 & 2
\end{array}\right)\\
A^{\gamma_2} &=
\left(\begin{array}{rrr}
1 & 2 & 2 \\
2 & 4 & 3 \\
1 & 1 & 1
\end{array}\right)
\end{align*}
their characteristic polynomials are,
\begin{align*}
	P^{\gamma_1} &= x^{3} - 4 x^{2} + 1\\
	P^{\gamma_2} &= x^{3} - 6 x^{2} + 1,
\end{align*}
which discriminants are equal to,
\begin{align*}
	\Delta(P^{\gamma_1}) &= 229,\\
	\Delta(P^{\gamma_2}) &= 3^3 \cdot 31.
\end{align*}

\subsection*{Pinching and twisting}
For both algorithms, we establish a proof of the two following propositions.
\begin{proposition}\label{pinching}
	The matrices $A^{\gamma_1}$ and $A^{\gamma_2}$ are Galois-pinching,
\end{proposition}

\begin{proof}
Observe that their characteristic polynomials,
\begin{itemize}
	\item have their first and the last coefficients are 1 and -1, which implies according to the rational root theorem that they do not have a rational root, and thus are irreducible over $\mathbb Q$;
	\item have positive discriminants, thus all of their roots are real;
	\item moreover the discriminants are not a square of a rational number, this implies (see \cite{Conrad}) that their Galois groups are isomorphic to $S_3$.
\end{itemize}
\end{proof}

\begin{proposition}\label{twisting}
	The matrices $A^{\gamma_1}$ and $A^{\gamma_2}$ do not share a common proper invariant subspace.
\end{proposition}
\begin{proof}
	This comes from the fact that $\Delta(P^{\gamma_1}) \wedge \Delta(P^{\gamma_2}) = 1$ as in the proof of Theorem 5.4 in \cite{MatheusMollerYoccoz15}.
\end{proof}

Proposition \ref{pinching} and Proposition \ref{twisting} imply that the cocycle is pinching and twisting in the sense of \cite{AvilaViana07}.
Now Theorem \ref{main} follows using bounded distortion property for the measure and Theorem 7.1 in \cite{AvilaViana07}.

\begin{remark*}
Formally, Theorem 2.17 in \cite{MatheusMollerYoccoz15} is stated for the symplectic cocycles while we work with a case of a group $SL(3,\mathbb{Z})$. However, it is easy to check that the proof of the same statement in our case follows the same strategy but is significantly more simple. Indeed, the proof of the pinching property of Galois-pinching matrix repeats verbatim the proof of Lemma 4.2 in \cite{MatheusMollerYoccoz15}, and the main goal is to prove an analogue of Theorem 4.6. Lemma 4.8 again works in case of $SL(n, \mathbb{Z})$, so $k$-twisting is equivalent to the property that graphs $\Gamma_k(B)$ are complete. But this last statement is much easier to prove in our case:  since $B$ is invertible, each $\Gamma_k(B)$ contains an arrow; moreover, by definition $\Gamma_k(B)$ is invariant under the action of the Galois group; but in case of $SL(3,\mathbb Z)$ Galois group acts by all possible permutation, hence from a given arrow we can  get all the arrows and completeness of the graph follows.
\end{remark*}

\bibliographystyle{plain}
\bibliography{biblio}

\end{document}